\documentclass[10pt]{article}
\usepackage{epsfig}
\usepackage{amssymb,amsmath,amsthm,amscd}
\usepackage{latexsym}
\usepackage{xcolor,bm}
\usepackage{xcolor,bm}
\usepackage{cite}
\usepackage{color} 
 \newtheorem{thm}{Theorem}[section]
 \newtheorem{lem}{Lemma}[section]
  
  \newtheorem{rem}{Remark}[section]
 \newtheorem{defn}{Definition}[section]

\makeatletter
\newcommand\dlmu[2][14.5cm]{\hskip0pt\underline{\hb@xt@ #1{\hss#2\hss}}\hskip1pt}
\makeatother

\title{\textsf{ Existence and invariant measure of pullback attractors for 3D Navier-Stokes-Voigt equations with delay}
}

\author{
\textsf{Yuming Qin}$^{a,b}$\footnote{Corresponding author E-mail:
yuming@dhu.edu.cn}\,,\,
\qquad
\textsf{Huite Jiang}$^a$\footnote{E-mail: huite$ \_ $Jiang@mail.dhu.edu.cn}\,,\,
\qquad
\\\\
 \small\it \small\it $^a$ Department of Mathematics, Donghua University,
\small\it Shanghai 201620, P. R. China,\\
\small\it $^b$ Institute for Nonlinear Science, Donghua University, Shanghai 201620, P. R. China.\\
}

\date{\small}

\begin{document}

\baselineskip 18pt

\maketitle\dlmu{}

\begin{abstract}

In this paper, we study the long-time dynamics of 3D non-autonomous Navier-Stokes-Voigt(NSV) equations with delay.
Inspired by \cite{SCD07}, we use the contractive function method to prove the pullback $\mathcal{D}$-asymptotical compactness and
existence of the pullback attractors.
Furthermore, we verify the regularity of pullback attractors by the method in \cite{CP09, WZL18, ZS15} and there exists a unique family of Borel invariant probability measures which is supported by the pullback attractors. 



\vskip 3mm

\noindent\textbf{Keywords}:
Navier-Stokes-Voigt equations;
regularity;
pullback attractors;
invariant measure

\noindent\textbf{AMS Subject Classifications}: 35Q30; 35Q35; 37N10; 35B40; 35B41



\end{abstract}

\baselineskip 20pt

\setcounter {equation}{0}
\section{Introduction}

\quad \,  In this paper, we investigate the long-time dynamics for the following 3D incompressible non-autonomous Navier-Stokes-Voigt(NSV) equations with delay 
\begin{align}\label{1.1}
\left\{\begin{array}{l}
\partial_t u-\nu \triangle u-\alpha^2 \triangle u_t+(u \cdot \nabla) u+\nabla p
=f(t)+g\left(t, u_t\right), \text { in }(\tau,+\infty) \times \Omega, \\
\nabla \cdot u=0, \quad \text { in }(\tau,+\infty) \times \Omega, \\
u=0, \quad \text { on }(\tau,+\infty) \times \partial \Omega, \\
u(\tau, x)=u_\tau(x), \quad \text { in } \Omega, \\
u(\tau+t, x)=\varphi(t, x), \quad \text { in }(-h, 0) \times \Omega,
\end{array}\right.
\end{align}
where
$\Omega \subset \mathbb{R}^3$ represents a bounded domain with a smooth boundary denoted by $\partial \Omega$,
the velocity vector field is expressed as $u=u(t, x)=(u_1(t, x), u_2(t, x), u_3(t, x))$, with $u_\tau$ representing the initial velocity field at the initial time $\tau \in \mathbb{R}$,
the variable $p$ denotes pressure, $\nu>0$ stands for kinematic viscosity, $\alpha$ serves as the length scale characterizing the fluid's elasticity,
$\varphi$ is a given function defined in the interval $(-h, 0)$, and $f$ is an external force term which may depend on time.
Here, $u_t$ denotes a segment of the solution, in other words, given $h>0$ and a function $u:[\tau-h,+\infty) \times \Omega \rightarrow \mathbb{R}$, for each $t \geq \tau$, we define the mapping $u_t:[-h, 0] \times \Omega \rightarrow \mathbb{R}$ by
$$
u_t(\theta, x)=u(t+\theta, x), \text { for } \theta \in[-h, 0], \enskip x \in \Omega.
$$

Navier-Stokes-Voigt systems have been proposed as regularizations of the 3D Navier-Stokes equation for the purpose of direct numerical simulation  in both autonomous and non-autonomous cases (cf. \cite{BB12, BKR16, BS17, CLT06, CR01, EHL13, KLR12, LR13, N16}).
The delay term $g\left(t, u_t\right)$ represents, for instance, the influence of an external force with some kind of delay, memory or hereditary characteristics, although can also model some kind of feedback controls.
There are several kinds of delay terms in the equations researched recently, where delay terms have three typical kinds.
The first type involves a constant delay or discrete delay expressed as $g(u(x, t-h))$, where $h$ remains a positive constant, and $g$ represents a specified functional. The second type comprises a variable delay $g(u(x, t-\tau(t)))$, where $\tau$ denotes a function dependent on $t$. The third type features a distributed delay formulated as $\int_{-h}^0 \mu(t-s) g(u(x, s)) d s$, incorporating the kernel function $\mu$.
Besides, the memory also have different types, such as general hereditary memory in the form of $\int_0^{\infty} \kappa(s) \Delta u(t-s) \mathrm{d} s$, memory effected by viscoelasticity in the form of $\int_0^{\infty} \mu(t-s) \Delta^2 u(t-s) \mathrm{d} s$ and so on.

The Navier-Stokes-Voigt equations and their related equations with finite delays or memory have been studied recently in some particular cases for the delay \cite{GM20, T21, SY23, N45, AT18}.
In \cite{T21}, the authors proved existence of polynomial stationary solutions by structuring appropriate Lyapunov functional in a special case of unbounded variable delay with a sufficient condition of parameters.
In \cite{N45}, the authors verified the existence of uniform attractors with damping and memory in the form of $\int_0^{\infty} \kappa(s) \Delta u(t-s) \mathrm{d} s$.
Meanwhile, the above mentioned results extended and improved some existed results in 
\cite{YW20}.
Besides, Anh and  Thanh \cite{AT18} studied the existence of a compact global attractor for Navier-Stokes-Voigt equations with infinite delay $F(u_t)$, where $F(\xi)=\int_{-\infty}^0 G(s, \xi(s)) \mathrm{d} s$, and then the authors in \cite{AT18} obtained the existence and exponential stability of stationary solutions.
More recently, reference \cite{CMD20} studied the existence and uniqueness of stationary solutions by three methods for the equations with different types of delay terms while there are absence of existence of pullback attractors and geometrical structure analysis.
Nevertheless, to the best of our knowledge, there is a few references concerning the existence and invariant measure of the pullback attractors for 3D NSV equations with constant delays. 


The motivation of this article is to investigate the existence of invariant Borel probability measures for the 3D NSV equations with constant delays. The invariant measures and statistical solutions are very useful for us to understand the turbulence (see, e.g., \cite{BMR16, CG12, FMRT01, R09}). The key reason is that the measurements of several important aspects (such as mass and
velocity) of turbulent flows are actually measurements of time-average quantities. Statistical solutions have been introduced as a rigorous mathematical notion to formalize the object of ensemble average in the conventional statistical theory of turbulence.
Nowadays, there are two pervasive and complementary notions of statistical solutions for the Navier-Stokes equations. The one is the so-called Foias-Prodi statistical solutions introduced by Foias and Prodi in \cite{FP76}, which is a family of Borel probability measures
parameterized by the time variable and is defined on the phase space of the Navier-Stokes equations, representing the probability distribution of the velocity field of the flow at each time.
The other is the so-called Vishik-Fursikov statistical solutions given by Vishik and Fursikov in \cite{VF78}, which is a single Borel measure on the space of trajectories, representing the probability distribution of the space-time velocity field.
Recently,
the scholars validated the presence of pullback attractors concerning three-dimensional regularized Magnetohydrodynamics (MHD) equations and 3D non-autonomous globally modified Navier-Stokes
equations (see, e.g., \cite{ZY17, ZZ18}), along with a set of invariant Borel probability measures.
Later, reference \cite{WZT20} investigated the three-dimensional globally modified Navier-Stokes equations with unbounded variable delays.
This paper establishes not only the global well-posedness of solutions and the presence of a pullback attractors but also formulates a set of invariant Borel probability measures.
Motivated by the aforementioned references, we opt to explore the long-term behavior of 3D NSV equations featuring constant delays.

The first goal of this paper is to establish the existence of pullback attrators for three-dimensional Navier-Stokes-Voigt equations \eqref{1.1}. To this end, we need to obtain the existence of pullback absorbing sets and verify the pullback asymptotical compactness.
There is a lot of literature on how to prove the compactness of pullback attractors.
Ma, Wang and Zhong \cite{MWZ02} proposed the Condition(C) in autonomous case, then Wang and Zhong \cite{WZ08} generalized their results to the non-autonomous case. Later, Harraga and Yebdri applied the above mentioned methods to semilinear nonclassical diffusion equations with delay(see \cite[Proposition 3.9]{HY16}).
Another method to prove pullback asymptotical compactness is that norm convergence plus weak convergence leads to strong convergence.
References \cite{2GMR12, ZY17} used this method to verify the pullback $\mathcal{D}_\sigma$-asymptotical compactness of pullback attractors for three-dimensional non-autonomous NSV equations or globally modified Navier-Stokes
equations. The reason why the above methods are not applicable to our model is that affected by the delay term, we cannot prove $\mathop{\lim}\limits_{n \rightarrow \infty}\sup \|u(t; \tau_n, u_{\tau_n}, \varphi_n)\| \leq \|w_0\|$ to hold.
After we tried to use various methods to establish the pullback asymptotical compactness,
we finally determine to adopt contractive function method (see \cite{SCD07}) to verify the existence of pullback attractors.
One of the advantages of the method is that it does not need additional estimates resulting in increase of computational workload, but make full use of the existed estimates to testify the compactness.

The second goal of this paper is to establish the existence of a family of Borel
invariant probability measures for three-dimensional non-autonomous NSV equations \eqref{1.1}.
In recent years, a set of studies has introduced various methods to generate invariant measures for non-autonomous systems while imposing several assumptions on the dynamic process in question (refer to \cite{FMRT01, W09, L08, LR14}).
Presently, these results find some applications to formulating invariant measures for some evolution equations, as evident in works such as \cite{LS17, WZT20, ZY17, ZZ18} and the associated literature.
Distinctive variances exist between the autonomous dynamical system and its non-autonomous counterpart, particularly having evidents in the continuous dependency of dynamical systems on their parameters.
In pursuit of this objective, based on the outcomes presented by {\L}ukaszewicz and Robinson (\cite[Theorem 3.1]{LR14}), we will verify the continuity and boundedness nature of the function $\tau \longmapsto U(t, \tau)(u_\tau, \varphi)$, which takes values in $E_V^2$ over the interval $(-\infty, t]$ by exploiting the structure of the three-dimensional NSV equations(refer to the notation in Section 5).
The remaining sections of this paper are structured as outlined below.
In Sect. 2, we present various mathematical symbols, outline assumption conditions, and articulate the global results concerning weak solutions.
Sect. 3 is dedicated to demonstrating the existence of bounded absorbing sets, pullback $\mathcal{D}$-asymptotic compactness, and confirming the presence of pullback attractors within $E_V^2$.
In Sect. 4, an analysis of the regularity of pullback attractors is conducted. Finally, in Sect. 5, we establish the existence of an invariant measure associated with pullback attractors.

\section{Global Well-posedness of solutions}
\quad \,   In this paper, we use the following notation{\rm :}\\
$\mathbb{R}=$ the set of real numbers, $\mathbb{N}=$ the set of positive integers,
$\mathbb{R}_\tau=[\tau,+\infty), \quad \mathbb{R}_{+}=[0,+\infty)$;\\
$L^p(\Omega)=$ the 3D vector Lebesgue space with norm $\|\cdot\|_{L^p(\Omega)}$, in particular, $\|\cdot\|_{L^2(\Omega)}=\|\cdot\|$;\\
$H^m(\Omega)=$ the 3D Sobolev space $\{\phi=(\phi_1, \phi_2, \phi_3) \in L^2(\Omega), \nabla^k \phi \in L^2(\Omega), k \leq m\}$
with norm $\|\cdot\|_{H^m(\Omega)}$;\\
$H_0^1(\Omega)=\text { closure of }\{\phi: \phi=(\phi_1, \phi_2, \phi_3) \in (\mathcal{C}_0^{\infty}(\Omega))^3 \} \text { in } H^1(\Omega)$;\\
$\mathcal{V}=\{\phi \in (\mathcal{C}_0^{\infty}(\Omega))^3 : \phi=(\phi_1, \phi_2, \phi_3), \nabla \cdot \phi=0\}$;\\
$H=$ closure of $\mathcal{V}$ in $L^2(\Omega)$ with norm $\|\cdot\|$ and inner product $(\cdot, \cdot)$, $H^{\prime}=$ dual space of $H$;\\
$V=$ closure of $\mathcal{V}$ in $H_0^1(\Omega)$ with norm $\|\cdot\|_V$,
 $V^{\prime}=$ dual space of $V$;\\
$(\cdot, \cdot)=$ the inner product in $H$, $\langle\cdot$, $\cdot\rangle=$ the dual pairing between $V$ and $V^{\prime}$;\\
$\mathcal{C}_H=\mathcal{C}([-h, 0] ; H)$, $\mathcal{C}_V=\mathcal{C}([-h, 0] ; V)$ ($h$ is a fixed positive number );\\
$L_H^2=L^2(-h, 0 ; H)$, $L_V^2=L^2(-h, 0 ; V)$,
$E_H^2=H \times L_H^2$, $E_V^2=V \times L_V^2$; \\ 
$E_H^2$ and $E_V^2$ are Hilbert spaces with the norm respectively
$$
\|(u_0, \varphi)\|^2_{E_H^2}=\|u_0\|^2+\int_{-h}^0 \|\varphi(\theta)\|^2 d \theta, \quad (u_0, \varphi) \in E_H^2,
$$
and
$$
 \|(u_0,  \varphi)\|_{E_V^2}
=\|\nabla u_0\|^2+ \int_{-h}^0 \|\nabla \varphi(\theta)\|^2 d \theta, \quad(u_0, \varphi) \in E_V^2.
$$


For short, we next introduce some operators and properties (see, e.g., \cite{T79}) to put our problem into an abstract framework.
Firstly, we denote $A: V \rightarrow V^{\prime}$ as
$$
\langle A u, v\rangle=(\nabla u, \nabla v), \quad \forall u, v \in V .
$$
Denote $D(A)=H^2(\Omega)\cap V$, then $A w=-P(\Delta w)$, $\forall w \in D(A)$, is the the Stokes operator, where $P$ is the Leray-Helmholtz operator from $L^2(\Omega)$ onto $H$.
Naturally, we further define space $E_{D(A)}^2=D(A) \times L^2_{D(A)}$ with the norm
$$
\|(u_0,\varphi)\|^2_{E_{D(A)}^2}
=\|\triangle u_0\|^2+\int_{-h}^0 \|\triangle \varphi(\theta)\|^2 d \theta, \quad(u_0, \varphi) \in E_{D(A)}^2.
$$
and $\|\cdot\|_{E_{D(A)}^2}$ is the norm of space $E_V^2$.

Secondly, we consider the trilinear form defined as
$$
b(u, v, w)=\sum_{i, j=1}^3 \int_{\Omega} u_i \frac{\partial v_j}{\partial x_i} w_j d x, \quad \forall u, v, w \in H_0^1.
$$
In particular, $b(u, v, w)$ can be extended continuously to make sense for all $u, v, w \in V$, and is a continuous trilinear form on $V \times V \times V$, and satisfies
\begin{align}
b(u, v, w)=-b(u, w, v), \, b(u, v, v)=0, \quad \forall u, v, w \in V. \label{2.3}
\end{align}
On the other hand, for any $u \in V$, we will use $B(u)$ to denote the element of $V^{\prime}$ given by
$$
\langle B(u), w\rangle=b(u, u, w), \quad \forall w \in V .
$$
For above introduced operators, we select the following estimations (cf \cite{T79}).
There exist some positive constants $C_i(i=1,2,3,4,5)$ depending only on $\Omega$ such that
\begin{align}
&
||b(u, v, w)|| \leq C_1\|u\|_{V}\|v\|_{V}\|w\|_{V}, \quad \forall u, v, w \in V, \label{2.2}\\
&
||b(u, v, w)|| \leq C_2||A u||^{1 / 2}\|u\|_{V}^{1 / 2}\|v\|_{V}||w||, \enskip \forall \, (u, v, w) \in D(A) \times V \times H, \label{2.4}\\
&
||b(u, v, w)|| \leq C_3||u||^{1 / 2}||u||_{V}^{1 / 2}||v||_{V}^{1 / 2}||Av||^{1 / 2}||w||, \enskip \forall \, (u, v, w) \in V \times D(A) \times H,\label{2.5}\\
&
\|B(u)\|_{V^{\prime}} \leq C_4\|u\|_{V}^2, \quad \forall u \in V,\label{2.6}\\
&
||B(u)|| \leq C_5||A u||^{1 / 2}\|u\|_{V}^{3 / 2}, \quad \forall u \in D(A) . \label{2.7}
\end{align}

To ensure the well-posedness of problem \eqref{1.1},
we need to give some suitable assumptions on the external force $f$ and the delay term $g$.
Assume $g: \mathbb{R} \times C_V \rightarrow L^2(\Omega)$, satisfying:
\begin{enumerate}



\item[{\rm \textbf{(\textsf{H1})}}] For all fixed $\xi \in C_V$, $g(\cdot, \xi)$ is measurable.

\item[{\rm \textbf{(\textsf{H2})}}] $g(t, 0)=0, \, \forall \, t \in \mathbb{R}$.

\item[{\rm \textbf{(\textsf{H3})}}] There exists $L_g>0$ such that for all $t \geq \tau$ and $ \forall \, \xi$, $\mu \in C_V$,
$$
\|g(t, \xi)-g(t, \mu)\| \leq L_g\|\xi-\mu\|_{C_V}.
$$

\item[{\rm \textbf{(\textsf{H4})}}] There exist positive constants $C_g \in \big(0, (2\nu-\sigma\lambda_1^{-1}-\alpha^2\sigma)(\lambda_1^{-1}+1)^{-1}C_6^{-\frac12}\big)$ and $\sigma \in \big(0, (\lambda_1^{-1}+\alpha^2)^{-1}(2\nu-4C_gC_6^{\frac12}(\lambda_1^{-1}+1))\big)$ such that, for all $\tau \leq t<T$ and $u, v \in C([\tau-h, T]$; $V)$,
$$
\int_\tau^t e^{\sigma s}\|g(s, u_s)-g(s, v_s)\|^2 d s \leq C_g^2 \int_{\tau-h}^t e^{\sigma s}\|u(s)-v(s)\|_{V}^2 d s .
$$

\end{enumerate}

By the above notation and analysis, we consider problem \eqref{1.1} excluding the pressure $p$, in the solenoidal vector field as
\begin{align}\label{91}
\left\{\begin{array}{l}
\text { For each } \tau \in \mathbb{R} \enskip \text{and} \enskip \forall \, T>\tau, \,\text{find} \\
u \in L^2(\tau-h, T ; V) \cap L^{\infty}(\tau, T ; V) \\ \text { such that} \enskip
\frac{d}{d t}(u(t)+\alpha^2 A u(t))+v A u(t)+B(u(t))=f(t)+g(t, u_t), \\
\text { in } \enskip \mathcal{D}^{\prime}(\tau, \infty ; V^{\prime}), \, u(\tau)=u_\tau, \, u(\tau+t)=\varphi(t), \enskip t \in(-h, 0).
\end{array}\right.
\end{align}

The existence and uniqueness of weak solutions $u$ to problem \eqref{1.1} can be obtained by the usual Faedo-Galerkin approximation and a compactness method \cite{CMD20}. We state the main results and omit the details here.

\begin{thm}\label{thm2.1}
Assume for each $\tau \in \mathbb{R}$, $f \in L_{l o c}^2(\mathbb{R} ; V^{\prime})$, $(u_\tau, \varphi) \in E_V^2$ and $g: \mathbb{R} \times C_V \rightarrow L^2(\Omega)$ satisfying \text{(H1)-(H4)}. It is said that $u$ is a weak solution to \eqref{1.1} if and only if $u \in C([\tau,+\infty) ; V)$, $\frac{d u}{d t} \in L^2(\tau, T ; V)$ for all $T>\tau$, and
\begin{align*}
&
u(t)+\alpha^2 A u(t)  +\int_\tau^t(v A u(s)+B(u(s))) d s \\
&
\quad =u^\tau+\alpha^2 A u^\tau+\int_\tau^t f(s) d s+\int_\tau^t g(s, u_s) d s \text { (equality in } V^{\prime}),
\end{align*}
for all $t \geq \tau$.
\end{thm}

\begin{thm}\label{thm2.2}
Assume for each $\tau \in \mathbb{R}$, $f \in L_{l o c}^2(\mathbb{R} ;H)$, $(u_\tau, \varphi) \in E_{D(A)}^2$ and $g: \mathbb{R} \times C_V \rightarrow L^2(\Omega)$ satisfying \text{(H1)-(H4)}.  It is said that $u$ is a weak solution to \eqref{1.1} if and only if $u \in C([\tau, \infty) ; D(A))$ and $\frac{d u}{d t} \in L^2(\tau, T ; D(A))$ for all $T>\tau$, and
\begin{align*}
&
\frac{1}{2} \frac{d}{d t}\left(\|u(t)\|^2+\alpha^2|A u(t)|^2\right)+v|A u(t)|^2+(B(u(t)), A u(t)) \\
&
\quad =\left(f(t)+g\left(t, u_t\right), A u(t)\right), \text { a.e. } t > \tau .
\end{align*}
\end{thm}

\section{Existence of pullback attractors in $E_V^2$ norm}

\quad \,  In this section, we will prove the existence of pullback attractors of the system \eqref{1.1}. Before proving the main results, we first recall some necessary abstract concepts in this section, such as basic definitions and properties of function spaces and attractors.


Denote by $\mathcal{P}(E_V^2)$ the family of all non-empty subsets of $E_V^2$ and
let be given $\mathcal{D}$ a non-empty class of families parameterized in time $\widehat{D}=\{D(t): t \in$ $\mathbb{R}\} \subset \mathcal{P}(E_V^2)$. The class $\mathcal{D}$ will be called a universe in $\mathcal{P}(E_V^2)$.

\begin{defn}  (\cite{R11})
The Hausdorff semi-distance between two nonempty sets $X_a, X_b \subset E_V^2$ is defined by
$$
\operatorname{dist}_{E_V^2}(X_a, X_b)=\sup _{x_1 \in X_a} \inf _{x_2 \in X_b}\|x_1-x_2\|_{E_V^2} .
$$
\end{defn}

\begin{defn} 
(\cite{CLR06, 2CLR06, MR09})
A process $U$ on $E_V^2$ is a mapping $\{(t, \tau) \in \mathbb{R}^2: \tau \leqslant t\}=\mathbb{R}_d^2 \times E_V^2 \ni (t, \tau, x) \mapsto U(t, \tau)x \enskip \text{in} \enskip E_V^2$ satisfies that $U(\tau, \tau) u=u$ for any $(\tau, x) \in \mathbb{R} \times E_V^2$ and $U(t, s) U(s, \tau)=U(t, \tau)$ for all $t \geq s \geqslant \tau$ and all $x \in X$.
\end{defn}




\begin{defn} 
(\cite{PSZ18, ZS15})
Let $\widehat{D}=\{D(t)\}_{t \in \mathbb{R}}$ be a family of bounded subsets in $E_V^2$. A process $U(\cdot, \cdot)$ is said to be pullback $\mathcal{D}$-asymptotically compact, if for any $t \in \mathbb{R}$, any sequence $\tau_n \rightarrow-\infty$ and $x_n \in D(\tau_n)$, the sequence $\{U(t, \tau_n) x_n\}_{n \in \mathbb{N}}$ is precompact in $E_V^2$.
\end{defn}

\begin{defn}(\cite{GMR12, SCD07}) \label{defn2.6}
Suppose the set $B \subset E_V^2$, then a function $\psi(\cdot, \cdot)$ defined on $E_V^2 \times E_V^2$ is said to be a contractive function on $B \times B$, if for any sequence $\{x_n\}_{n=1}^{\infty} \subset B$, there is a subsequence $\{x_{n_k}\}_{k=1}^{\infty} \subset \{x_n\}_{n=1}^{\infty}$ such that
$$
\lim _{k \rightarrow \infty} \lim _{l \rightarrow \infty} \psi(x_{n_k}, x_{n_l})=0 .
$$
For simplicity, we denote the set of all contractive functions on $B \times B$ by $\mathcal{C}(\widehat{B})$.
\end{defn}

\begin{lem}(\cite{SCD07}) \label{lem2.1}
Assume the process $\{U(t, \tau)\}_{t \geq \tau}$ has a pullback $\mathcal{D}$-absorbing set $\widehat{B}=\{B(t): t \in$ $\mathbb{R}\}$ on $E_V^2$ and there exist $T=T(t, \widehat{B}, \tilde{c})=t-\tau$ and $\psi_{t, T}(\cdot, \cdot) \in \mathcal{C}(\widehat{B})$ such that
$$
\|U(t, t-T) x-U(t, t-T) y\|_X \leq \tilde{c}+\psi_{t, T}(x, y),
$$
for any $x, y \in B(\tau)$ and $\tilde{c}>0$, then $\{U(t, \tau)\}_{t \geq \tau}$ is pullback $\mathcal{D}$-asymptotically compact on $E_V^2$.
\end{lem}




\begin{defn} \label{defn2.7}
(\cite{PSZ18, ZS15})
A family $\mathcal{A}_{\mathcal{D}}=\{\mathcal{A}_{\mathcal{D}}(t): t \in \mathbb{R}\} \subset \mathcal{P}(E_V^2)$ is said to be a pullback $\mathcal{D}$-attractor for the process $\{U(t, \tau)\}_{t \geq \tau}$ on $E_V^2$ if the following properties hold:\\
(i) the set $\mathcal{A}_{\mathcal{D}}(t)$ is compact in $E_V^2$ for any $t \in \mathbb{R}$;\\
(ii) $\mathcal{A}_{\mathcal{D}}$ is pullback $\mathcal{D}$-attracting in $E_V^2$, i.e.,
$$
\lim _{\tau \rightarrow-\infty} \operatorname{dist}_{E_V^2}(U(t, \tau) D(\tau), \mathcal{A}_{\mathcal{D}}(t))=0,
$$
for any $\widehat{D} \in \mathcal{D}$ and $t \in \mathbb{R}$;\\
(iii) $\mathcal{A}_{\mathcal{D}}$ is invariant, i.e., $U(t, \tau) \mathcal{A}_{\mathcal{D}}(\tau)=\mathcal{A}_{\mathcal{D}}(t)$, for any $\tau \leq t$. 
\end{defn}

\begin{lem}\label{lem3.1}
Assume $f \in L_{l o c}^2(\mathbb{R} ;V^{\prime})$, $(u_\tau, \varphi) \in E_V^2$ and $g: \mathbb{R} \times C_V \rightarrow L^2(\Omega)$ satisfies \text{(H1)-(H4)}. Then the family of mappings $U(t, \tau): E_V^2 \rightarrow E_V^2 $,
\begin{align}\label{4.9}
(u_\tau, \varphi) \mapsto U(t, \tau)(u_\tau, \varphi)=(u(t),u_t),
\end{align}
with $(t, \tau) \in \mathbb{R}^2$ and $u$ being the weak solution to \eqref{1.1}, defines a continuous process.
\end{lem}

\begin{proof}
It can be immediately obtained by Theorem \ref{thm2.1}, and one can refer to
\cite[Theorem 3.3]{CMD20} for the details.
\end{proof}

For the obtention of a pullback absorbing family for the process $\{U(t, \tau)\}_{t \geq \tau}$, we have the following lemma.

\begin{lem} \label{lem4.1}
Assume $f \in L_{l o c}^2(\mathbb{R} ; V^{\prime})$, $(u_\tau, \varphi) \in E_V^2$ and $g: \mathbb{R} \times C_V \rightarrow L^2(\Omega)$ satisfies \text{(H1)-(H4)}. Then, for any
\begin{align}\label{4.10}
0<\sigma<(\lambda_1^{-1}+\alpha^2)^{-1}(2\nu-4C_gC_6^{\frac12}(\lambda_1^{-1}+1)),
\end{align}
the solution $u=u(\cdot ; \tau, u_\tau, \varphi)$ of \eqref{1.1} satisfies
\begin{align}\label{4.11}
\begin{aligned}
\|\nabla u(t)\|^2+\eta_1 e^{- \sigma t}
\int_\tau^t e^{\sigma s}\|\nabla u(s)\|^2 \mathrm{d} s
\leq&
\alpha^{-2}(\lambda^{-1}+\alpha^2+2C_gC_6^{\frac12})e^{-\sigma(t-\tau)}||(u_\tau,\varphi)||^2_{E_V^2}\\
&+
\frac{1}{\beta \alpha^2}e^{-\sigma t} \int_\tau^t e^{\sigma s}||f(s)||^2_{V^{\prime}}\mathrm{d} s,
\end{aligned}
\end{align}
\begin{align}\label{4.12}
\begin{aligned}
\int_{t-h}^t||\nabla u(s)||^2\mathrm{d} s
\leq&
\eta_1^{-1}\big\{\alpha^{-2}(\lambda^{-1}+\alpha^2+2C_gC_6^{\frac12})e^{-\sigma(t-\tau-h)}||(u_\tau,\varphi)||^2_{E_V^2}\\
&+
\frac{1}{2\beta \alpha^2}e^{-\sigma (t-h)} \int_\tau^t e^{\sigma s}||f(s)||^2_{V^{\prime}}\mathrm{d} s\big\},
\end{aligned}
\end{align}
for all $t \geqslant \tau$, where
\begin{align}
&
\eta_1=\alpha^{-2}[2\nu-\sigma \lambda_1^{-1}-\alpha^2\sigma-(\beta+4C_gC_6^{\frac12})(\lambda_1^{-1}+1)]>0, \label{4.13}\\
&
\beta \in (0, (2\nu-\sigma\lambda_1^{-1}-\alpha^2\sigma)(\lambda_1^{-1}+1)^{-1}-4C_gC_6^{\frac12}).\label{4.14}
\end{align}
\end{lem}

\begin{proof}
Multiplying \eqref{1.1}$_1$ by $u \in V$ and integrating the $\Omega$, we have
\begin{align}\label{4.14}
\frac12 \frac{\mathrm{d}}{\mathrm{d}t}(\|u(t)\|^2+\alpha^2\|\nabla u(t)\|^2) +\nu \|\nabla u(t)\|^2 =\langle f(t),u(t)\rangle_{V^{\prime} \times V}+\langle g(t,u_t),u(t)\rangle_{V^{\prime} \times V}.
\end{align}
Furthermore, multiplying \eqref{4.14} by $e^{\sigma t}$, and then integrating in the interval $[\tau, t]$, we infer from the Poincar\'{e} inequality
\begin{align}\label{4.15}
\begin{aligned}
e^{\sigma t}(\|u(t)\|^2+\alpha^2 \|\nabla u(t)\|^2)
\leq&
(\lambda_1^{-1}+\alpha^2) e^{\sigma \tau}\|\nabla u_\tau\|^2+(\sigma \lambda_1^{-1}+\alpha^2 \sigma-2 \nu)\int_\tau^t e^{\sigma s}\|\nabla u(s)\|^2\mathrm{d}s \\
&+
2\int_\tau^te^{\sigma s}\langle f(s),u(s)\rangle_{V^{\prime} \times V}\mathrm{d} s+2\int_\tau^t e^{\sigma s}\langle g(s, u_s),u(s)\rangle_{V^{\prime} \times V}\mathrm{d}s.
\end{aligned}
\end{align}
Using the Young inequality, we can derive
\begin{align}\label{4.16}
2e^{\sigma t}\langle f(t),u(t)\rangle_{V^{\prime} \times V} \leq \beta e^{\sigma t}\|u(t)\|^2_V+\frac{1}{\beta}e^{\sigma t}\|f(t)\|^2_{V^{\prime}}.
\end{align}
Taking account of assumptions \text{(H2)-(H4)},
and combining that $(L^2(\Omega))^3 \hookrightarrow (H^{-1}(\Omega))^3$ is compact,
there exists a constant $C_6$(only depends on $\Omega$) such that
\begin{align}\label{4.17}
\begin{aligned}
2\int_\tau^t e^{\sigma s}\langle g(s, u_s),u(s)\rangle_{V^{\prime} \times V}\mathrm{d}s
&\leq
2CgC_6^{\frac12}e^{\sigma \tau}\|\varphi\|^2_{L_V^2}+4C_g C_6^{\frac12} \int_\tau^t e^{\sigma s} \|u(s)\|_V^2 \mathrm{d}s,
\end{aligned}
\end{align}
Inserting \eqref{4.15}-\eqref{4.17} to \eqref{4.14} and rearranging them briefly, we arrive at
\begin{align}\label{4.18}
\begin{aligned}
&
e^{\sigma t}(\|u(t)\|^2+\alpha^2 \|\nabla u(t)\|^2)\\
& \quad \leq
(\lambda_1^{-1}+\alpha^2+2C_gC_6^{\frac12}) e^{\sigma \tau}\|(u_\tau, \varphi)\|^2_{E^2_{V}}+\frac{1}{\beta}\int_\tau^t e^{\sigma s}\|f(s)\|_{V^{\prime}}^2\mathrm{d} s \\
& \quad \quad +
\big[\sigma \lambda_1^{-1}+\alpha^2 \sigma-2 \nu+(\beta+4C_gC_6^{\frac12})(\lambda_1^{-1}+1)\big]\int_\tau^t e^{\sigma s}\|\nabla u(s)\|^2\mathrm{d}s,
\end{aligned}
\end{align}
where $\sigma$ and $\beta$ are defined by \eqref{4.13} and \eqref{4.14}, respectively.

By \eqref{4.13}, then \eqref{4.18} yields
\begin{align}\label{4.20}
\begin{aligned}
&
\|\nabla u(t)\|^2+ \eta_1\int_\tau^t e^{\sigma s}\|\nabla u(s)\|^2\mathrm{d}s\\
& \quad \leq
\alpha^{-2}(\lambda_1^{-1}+\alpha^2+2C_gC_6^{\frac12})e^{-\sigma(t-\tau)}\|(u_\tau,\varphi)\|_{E^2_V}^2
+\frac{1}{\beta\alpha^2}e^{-\sigma t}\int_\tau^t e^{\sigma s}\|f(s)\|^2_{V^{\prime}} \mathrm{d}s.
\end{aligned}
\end{align}
Furthermore, for any $t-h \geq \tau$, we notice that
\begin{align}\label{4.22}
\int_\tau^t e^{\sigma s}\|\nabla u(s)\|^2\mathrm{d}s \geq
\int_{t-h}^t e^{\sigma s}\|\nabla u(s)\|^2\mathrm{d}s
\geq
e^{\sigma(t-h)}\int_{t-h}^t e^{\sigma s}\|\nabla u(s)\|^2\mathrm{d}s.
\end{align}
Thus, from \eqref{4.20}, we obtain
\begin{align}\label{4.23}
\begin{aligned}
\int_{t-h}^t\|\nabla u(s)\|^2 \mathrm{d}s
\leq&
\eta_1^{-1}\big\{\alpha^{-2}(\lambda_1^{-1}+\alpha^2+2C_gC_6^{\frac12})e^{-\sigma(t-\tau-h)}\|(u_\tau,\varphi)\|_{E^2_V}^2\\
&+
\frac{1}{\beta\alpha^2}e^{-\sigma (t-h)}\int_\tau^t e^{\sigma s}\|f(s)\|^2_{V^{\prime}} \mathrm{d}s\big\}.
\end{aligned}
\end{align}
The proof is complete.
\end{proof}

\begin{defn}\label{defn4.1}
Let
\begin{align}
\mathcal{R}_\sigma:=\{\rho_\sigma(t): \mathbb{R}\mapsto (0,+ \infty)| \mathop{\lim}_{t \rightarrow -\infty}e^{\sigma t}\rho^2_\sigma(t)=0\}.
\end{align}
By $\mathcal{D}^{V}_\sigma$ we denote the class of all families $\widehat{D}^{V}=\{D^{V}(t): t \in \mathbb{R}\} \subset \mathcal{P}({E^2_V})$ such that $D^V(t) \subset \overline{\mathcal{B}}_{E^2_V}(0, \rho_{\widehat{D}^{V}}(t))$, for some $\rho_{\widehat{D}^{V}} \in \mathcal{R}_\sigma$, where $\overline{\mathcal{B}}_{E^2_V}(0, \rho_{\widehat{D}^{V}}(t))$ denotes the closed ball in $E_V^2$ centered at 0 with radius $\rho_{\widehat{D}^{V}}(t)$.
\end{defn}


Based on the above results, adding some suitable conditions to the function $f$ of problem \eqref{1.1}, then we obtain the existence of the $\mathcal{D}^{E^2_V}_\sigma$-absorbing family of process $\{U(t, \tau)\}_{t \geq \tau}$ on $E_V^2$.

\begin{lem}\label{lem4.2}
Assume $f \in L_{l o c}^2(\mathbb{R} ; V^{\prime})$ satisfying
\begin{align}\label{4.26}
\int^t_{-\infty} e^{\sigma s}\|f(s)\|^2_{V^{\prime}} \mathrm{d}s < + \infty, \enskip \forall t \in \mathbb{R},
\end{align}
$(u_\tau, \varphi) \in E_V^2$ and $g: \mathbb{R} \times C_V \rightarrow L^2(\Omega)$ satisfies \text{(H1)-(H4)}, the family $\widehat{B}^{V}$ 
given by
\begin{align}\label{4.27}
B^{V}(t):=\big\{(u(t), u_t) \in E_V^2 \, \big| \, \|(u(t), u_t)\|_{E_V^2}^2 \leq R_1(t),\|\frac{d u_t(s)}{d s}\|_{L^2((-h, 0) ; V)}^2 \leq R_2(t)\big\},
\end{align}
is pullback $\mathcal{D}_\sigma^{E^2_V}$-absorbing for the process $\{U(t, \tau)\}_{t \geq \tau}$ defined by \eqref{4.9}, where
\begin{align}
&
R_1^2(t)=(1+\eta_1^{-1} e^{\sigma h}) \rho_\sigma(t), \label{4.28}\\
&
R^2_2(t)=1+c e^{2h\sigma}\rho_\sigma(t)+c\|f\|^2_{L^2([t-h, t]; V^{\prime})}+c \mathop{\max}_{s \in [t-2h, t]}\|\nabla u(s)\|^4, \label{4.29}\\
&
\rho_\sigma(t)=ce^{-\sigma t} \int_{-\infty}^t e^{\sigma s}\|f(s)\|^2_{V^{\prime}} \mathrm{d} s,\label{4.30}
\end{align}
where $c >0$ is a constant that depends only on $\alpha$, $\lambda_1$, $C_g$ and $\nu$.
\end{lem}

\begin{proof}
First, we observe that for all $t \in \mathbb{R}$,
$$
B^{V}(t)\subset  \big\{U(t, \tau)(u_\tau, \varphi)=(u(t),u_t) \in E_V^2:\|(u(t), u_t)\|_{E_V^2} \leq R_1(t)\},
$$
with
$$
\mathop{\lim}_{t \rightarrow - \infty} e^{\sigma t}R_1^2(t)=0,
$$
and so $\widehat{B}^{V} \in \mathcal{D}^{E^2_V}_\sigma$.
Now, we prove the $U(t,\tau)D^{V}(\tau) \subset B^{V}(t)$, for all $\tau \leq t$. To do this, we proceed in two steps.

{\bf {Step 1:}}  This step concerns the asymptotic estimate using $R_1(t)$ for $t \in \mathbb{R}$, fixed. It may be proved as follows.
For any $\widehat{D}^{V}=\{D^{V}(t): t \in \mathbb{R}\} \in \mathcal{D}^{V}_\sigma$ and $(u_\tau, \varphi) \in D^{V}(\tau)$, we set $u(\cdot):=u(\cdot,; \tau, u_\tau, \varphi)$.
By the definition, we have
\begin{align}\label{4.31}
\|U(t,\tau)(u_\tau, \varphi)\|^2_{E_V^2}=\|\nabla u(s)\|^2+\int_{-h}^0\|\nabla \varphi(s)\|^2 \mathrm{d}s=\|\nabla u(s)\|^2+\int_{t-h}^t\|\nabla u(s)\|^2 \mathrm{d}s.
\end{align}
It follows from \eqref{4.12} that for any $t-h \geq \tau$,
\begin{align}
\begin{aligned}
\int_{t-h}^t||\nabla u(s)||^2\mathrm{d} s
\leq&
\eta_1^{-1}\alpha^{-2}(\lambda^{-1}+\alpha^2+2C_gC_6^{\frac12})e^{-\sigma(t-\tau-h)}||(u_\tau,\varphi)||^2_{E_V^2}\\
&+
\frac{1}{2\beta \eta_1 \alpha^2}e^{-\sigma (t-h)} \int_{- \infty}^t e^{\sigma s}||f(s)||^2_{V^{\prime}}\mathrm{d} s,
\end{aligned}
\end{align}
for any $(u_\tau, \varphi) \in E_V^2$.

We substitute the inequality and \eqref{4.11} in \eqref{4.31}, by the definition of $\rho_\sigma(t)$, and obtain
\begin{align}\label{94}
\begin{aligned}
\|U(t,\tau)(u_\tau, \varphi)\|^2_{E_V^2}
\leq&
ce^{-\sigma(t-\tau)}||(u_\tau,\varphi)||^2_{E_V^2}+ce^{-\sigma t} \int_{- \infty}^t e^{\sigma s}||f(s)||^2_{V^{\prime}}\mathrm{d} s\\
&+
c\eta_1^{-1}e^{-\sigma(t-\tau-h)}||(u_\tau,\varphi)||^2_{E_V^2}
+
c \eta_1^{-1}e^{-\sigma (t-h)} \int_{- \infty}^t e^{\sigma s}||f(s)||^2_{V^{\prime}}\mathrm{d} s\\
\leq&
c(1+\eta_1^{-1}e^{\sigma h}) e^{-\sigma(t-\tau)}||(u_\tau,\varphi)||^2_{E_V^2}
+(1+\eta_1^{-1}e^{\sigma h})\rho_\sigma(t)\\
\leq&
c(1+\eta_1^{-1}e^{\sigma h}) e^{-\sigma(t-\tau)}||(u_\tau,\varphi)||^2_{E_V^2}+R_1^2(t),
\end{aligned}
\end{align}
where $\rho_\sigma(t)$ and $R_1(t)$ are given by \eqref{4.28} and \eqref{4.30}, and $c >0$ is a constant that depends only on $\alpha$, $\lambda_1$, $C_g$ and $\nu$.
Hence,
\begin{align}\label{4.34}
\|U(t,\tau)(u_\tau, \varphi)\|^2_{E_V^2}\leq R_1^2(t),
\end{align}
as $e^{ \sigma \tau} \rightarrow 0$ when $\tau \rightarrow -\infty$.

{\bf {Step 2:}}  This step concerns the asymptotic estimate using $R_2(t)$. We assume that $t-2h \geq \tau$.
Multiplying \eqref{1.1}$_1$ by $\frac{\partial u}{\partial t}=\partial_t u \in V$ and integrating the resultant over $\Omega$, we derive
\begin{align}\label{4.35}
\begin{aligned}
&
\frac{\nu}{2}\frac{\mathrm{d}}{\mathrm{d}t}\|\nabla u(t)\|^2+\|\partial_t u(t)\|^2+\alpha^2 \|\nabla \partial_t u(t)\|^2\\
& \quad \leq
|\langle B(u(t)), \partial_t u(t)\rangle|_{V^{\prime} \times V}+|\langle f(t), \partial_t u(t)\rangle|_{V^{\prime} \times V}+|\langle g(t, u_t), \partial_t u(t)\rangle|_{V^{\prime} \times V}.
\end{aligned}
\end{align}
Besides, in light of \eqref{2.6},
the following inequalities hold
\begin{align}
| \langle B(u(t)),\partial_t u (t)\rangle |_{V^{\prime} \times V}
\leq&
C_4(\lambda_1^{-1}+1)^2\|\nabla u(t)\|^4 + \frac{C_4}{4}\|\partial_t u(t)\|^2_V, \label{4.36}\\
| \langle f(t),\partial_t u (t) \rangle |_{V^{\prime} \times V}
\leq&
\|f(t)\|_{V^{\prime}}^2+\frac14 \|\partial_t u(t)\|_V^2, \label{4.37}\\
| \langle g(t, u_t),\partial_t u \rangle |_{V^{\prime} \times V}
\leq&
\|g(t, u_t)\|^2_{V^{\prime}}+\frac14 \|\partial_t u\|_V^2
\leq
C_6\|g(t, u_t)\|^2+\frac14 \|\partial_t u\|_V^2. \label{4.38}
\end{align}
Meanwhile, by the equivalent property of the norm of space $V$, there exists a constant $C_7 >0$ satisfying $4C_7-C_4+2>0$ 
and $C_7^{\prime}$ such that
\begin{align}\label{4.39}
C_7||\partial_t u(t)||^2_{V} \leq \|\partial_t u(t)\|^2+\alpha^2\|\nabla \partial_t u(t)\|^2_V \leq C_7^{\prime}||\partial u(t)||^2_{V}.
\end{align}
Inserting \eqref{4.36}-\eqref{4.39} into \eqref{4.35}, we get the estimates
\begin{align}\label{4.40}
\begin{aligned}
\frac{\nu}{2} \frac{\mathrm{d}}{\mathrm{d}t}\|\nabla u(t)\|^2
\leq
C_4(\lambda_1^{-1}+1)^2\|\nabla u(t)\|^4+\frac{C_4-4C_7-2}{4}\|\partial_t u(t)\|^2_V
+\|f(t)\|^2_{V^{\prime}}+C_6\|g(t, u_t)\|^2.
\end{aligned}
\end{align}
Integrating over $[t-h, t]$, we notice that
\begin{align}\label{4.42}
\begin{aligned}
\frac{4C_7-C_4+2}{4} \int_{t-h}^t \|\partial_s u(s)\|^2_V \mathrm{d}s
\leq&
\frac{\nu}{2} \|\nabla u(t-h)\|^2+C_4(\lambda_1^{-1}+1)^2\int_{t-h}^t \|\nabla u(s)\|^4 \mathrm{d} s\\
&+
\int_{t-h}^t \|f(s)\|^2_{V^{\prime}} \mathrm{d} s+C_6\int_{t-h}^t \|g(s, u_s)\|^2 \mathrm{d} s.
\end{aligned}
\end{align}
From assumptions \text{(H2)}, \text{(H4)} and the Poincar\'{e} inequality, it follows
\begin{align}\label{4.43}
\int_{t-h}^t \|g(s, u_s)\|^2 \mathrm{d} s
\leq
C_g^2 \int_{t-2h}^t \|u(s)\|^2_V \mathrm{d} s
\leq
C_g^2(\lambda_1^{-1}+1)\int_{t-h}^t \|\nabla u(s)\|^2 \mathrm{d} s.
\end{align}
By this estimate and \eqref{4.42}, one has
\begin{align}\label{4.44}
\begin{aligned}
\frac{4C_7-C_4+2}{4} \int_{t-h}^t \|\partial_s u(s)\|^2_V \mathrm{d}s
\leq&
\frac{\nu}{2} \|\nabla u(t-h)\|^2+C_4(\lambda_1^{-1}+1)^2\int_{t-h}^t \|\nabla u(s)\|^4 \mathrm{d} s\\
&+
\int_{t-h}^t \|f(s)\|^2_{V^{\prime}} \mathrm{d} s+C_6\int_{t-h}^t \|g(s, u_s)\|^2 \mathrm{d} s\\
\leq&
\frac{\nu}{2} \|\nabla u(t-h)\|^2+C_6C_g^2(\lambda_1^{-1}+1)\int_{t-2h}^t \|\nabla u(s)\|^2 \mathrm{d} s\\
&+
2C_4 h(\lambda_1^{-1}+1)^2 \mathop{\max}_{s \in [t-2h, t]}\|\nabla u(s)\|^4 +\|f\|^2_{L^2([t-h, t; V^{\prime}])}.
\end{aligned}
\end{align}
Now, we estimate $\|\nabla u(t-h)\|^2$. Replacing $t$ by $t-h$ in \eqref{4.11}, we immediately infer
\begin{align}\label{4.45}
\begin{aligned}
\|\nabla u(t-h)\|^2
\leq&
\alpha^{-2}(\lambda^{-1}+\alpha^2+2C_gC_6^{\frac12})e^{-\sigma(t-\tau-h)}||(u_\tau,\varphi)||^2_{E_V^2}\\
&+
\frac{1}{\beta \alpha^2}e^{-\sigma (t-h)} \int_\tau^t e^{\sigma s}||f(s)||^2_{V^{\prime}}\mathrm{d} s\\
\leq&
c e^{-\sigma(t-\tau-h)}||(u_\tau,\varphi)||^2_{E_V^2}
+c e^{-\sigma (t-h)} \int_\tau^t e^{\sigma s}||f(s)||^2_{V^{\prime}}\mathrm{d} s,
\end{aligned}
\end{align}
where $c >0$ is a constant that depends only on $\alpha$, $\lambda_1$, $C_g$ and $\nu$.
In view of $t-h \leq t$ and $e^{\sigma h} >1$, inequality \eqref{4.45} is improved to
\begin{align}\label{4.46}
\begin{aligned}
\|\nabla u(t-h)\|^2
\leq&
c e^{\sigma h} e^{-\sigma(t-\tau)}||(u_\tau,\varphi)||^2_{E_V^2}
+c e^{\sigma h} e^{-\sigma t} \int_\tau^t e^{\sigma s}||f(s)||^2_{V^{\prime}}\mathrm{d} s\\
\leq&
c e^{\sigma h} e^{-\sigma(t-\tau)}||(u_\tau,\varphi)||^2_{E_V^2}
+e^{\sigma h} \cdot \rho_\sigma(t).
\end{aligned}
\end{align}
Using \eqref{4.12} and taking $2h$ in place of $h$, we get that for any $\tau \leq t-2h$,
\begin{align}\label{4.48}
\begin{aligned}
\int^t_{t-2h} \|\nabla u(s)\|^2 \mathrm{d} s
\leq&
\eta_1^{-1}\alpha^{-2}(\lambda^{-1}+\alpha^2+2C_gC_6^{\frac12})e^{-\sigma(t-\tau-2h)}||(u_\tau,\varphi)||^2_{E_V^2}\\
&+
\frac{1}{\beta \eta_1 \alpha^2}e^{-\sigma (t-h)} \int_\tau^t e^{\sigma s}||f(s)||^2_{V^{\prime}}\mathrm{d} s\\
\leq&
ce^{-\sigma(t-\tau-2h)}||(u_\tau,\varphi)||^2_{E_V^2}
+
c e^{2 h \sigma } e^{-\sigma t} \int_{-\infty}^t e^{\sigma s}||f(s)||^2_{V^{\prime}}\mathrm{d} s\\
\leq&
ce^{-\sigma(t-\tau-2h)}||(u_\tau,\varphi)||^2_{E_V^2}
+
 e^{2 h \sigma } \cdot \rho_\sigma(t).
\end{aligned}
\end{align}
From \eqref{4.44}-\eqref{4.48}, we conclude that for all $t-2h \geq \tau$ and all $(u_\tau, \varphi) \in E_V^2$,
\begin{align}
\begin{aligned}
\int_{t-h}^t \|\partial_s u(s)\|^2_V \mathrm{d}s
\leq&
c e^{\sigma h} e^{-\sigma(t-\tau)}||(u_\tau,\varphi)||^2_{E_V^2}
+
c e^{2h \sigma} e^{-\sigma(t-\tau)}||(u_\tau,\varphi)||^2_{E_V^2}\\
&
+c e^{2h \sigma} \cdot \rho_\sigma(t)
+
c \mathop{\max}_{s \in [t-2h, t]}\|\nabla u(s)\|^4 +c \|f\|^2_{L^2([t-h, t; V^{\prime}])}.
\end{aligned}
\end{align}
Hence, for all $t-2h \geq \tau$ and for any $(u_\tau, \varphi) \in E_V^2$, we deduce that
\begin{align}
\int_{t-h}^t \|\partial_s u(s)\|^2_V \mathrm{d}s
\leq
c e^{\sigma h} e^{-\sigma(t-\tau)}||(u_\tau,\varphi)||^2_{E_V^2}
+
c e^{2h \sigma} e^{-\sigma(t-\tau)}||(u_\tau,\varphi)||^2_{E_V^2}
+R^2_2(t),
\end{align}
where $R_2^2(t)$ has been given in \eqref{4.29}.
We finally end up with
\begin{align}\label{4.52}
\int_{t-h}^t \|\partial_s u(s)\|^2_V \mathrm{d}s
\leq
R^2_2(t)
\end{align}
as $e^{\sigma \tau} \rightarrow 0$ when $\tau \rightarrow - \infty$.
Consequently, it is clear to see from \eqref{4.34}, \eqref{4.52} and the definition of $\mathcal{D}$ that the family $\widehat{B}^{V}$ given by \eqref{4.27} is the family of pullback $\mathcal{D}_\sigma^{E^2_V}$-absorbing sets for the process $\{U(t, \tau)\}_{t \geq \tau}$.
\end{proof}

Next, we will use the contractive function method to verify the existence of pullback $\mathcal{D}_\sigma^{E^2_V}$-attractor for the $\{U(t, \tau)\}_{t \geq \tau}$ of problem \eqref{1.1}.

\begin{lem}\label{lem4.3}
Under the assumptions of Lemma \ref{lem4.2}, assume $\varphi \in E_V^2$ is given, if $\{u^n(t)\}_{n \in \mathbb{N}^{+}}$is a sequence of the solutions to problem \eqref{1.1} with initial data $u^n(\tau) \in E_V^2$, then there exists a subsequence of $\{u^n(t)\}_{s \in \mathbb{N}^{+}}$that convergence strongly in $L^2(\tau, T ; V)$.
\end{lem}

\begin{proof}
It is a direct consequence of \cite[Theorem 3.3]{CMD20} and we omit details here.
\end{proof}

Now, we will establish the pullback $\mathcal{D}_\sigma^{E^2_V}$-asymptotic compactness for the process $\{U(t, \tau)\}_{t \geq \tau}$ of problem \eqref{1.1}.

\begin{lem}\label{lem4.4}
Under the assumptions of Lemma \ref{lem4.2}. Then the process $\{U(t, \tau)\}_{t \geq \tau}$ is pullback $\mathcal{D}_\sigma^{E^2_V}$-asymptotically compact in $E_V^2$.
\end{lem}

\begin{proof}
Suppose $(u^j(t), u_t^j)$ is a weak solution of problem \eqref{1.1}  corresponding to the initial value $u_\tau^j$, $\varphi^j(\theta, x) \in \widehat{D}^{V}$ ($j$=1,2), and assume $\tilde{u}=u^1-u^2$, $\tilde{\varphi}=\varphi^1-\varphi^2$, then we obtain
\begin{align}\label{4.53}
\begin{aligned}
\partial_t \tilde{u}(t)-\alpha^2 \triangle \partial_t \tilde{u}(t)-\nu \triangle \tilde{u}(t)+B(u^1(t), u^1(t))-B(u^2(t), u^2(t))=g(t, u_t^1(t))-g(t, u_t^2(t)),
\end{aligned}
\end{align}
with initial data
\begin{align}\label{4.54}
\tilde{u}^j(\tau+\theta, x)=\varphi^j(\theta, x), \enskip \theta \in [-h, 0], \, x \in \Omega.
\end{align}
Choosing $\tilde{u} \in V$ as a text function of \eqref{4.53}, we arrive at
\begin{align}\label{4.55}
\begin{aligned}
&
\frac12\frac{\mathrm{d}}{\mathrm{d}t}(\|\tilde{u}(t)\|^2+\alpha^2\|\nabla \tilde{u}(t)\|^2)+\nu \|\nabla \tilde{u}(t)\|^2+ \langle B(u^1(t))-B(u^2(t)), \tilde{u}(t)\rangle_{V^{\prime} \times V}\\
& \quad =
\langle g(t, u_t^1)- g(t, u_t^2), \tilde{u}(t) \rangle_{V^{\prime} \times V}.
\end{aligned}
\end{align}
Multiplying \eqref{4.55} by $e^{\sigma t}$,
we can estimate
\begin{align}\label{4.57}
\begin{aligned}
&
\frac{\mathrm{d}}{\mathrm{d}t}(e^{\sigma t}\|\tilde{u}(t)\|^2+\alpha^2e^{\sigma t}\|\nabla \tilde{u}(t)\|^2)\\
& \quad \leq
(\lambda_1^{-1}\sigma+\alpha^2 \sigma-2 \nu)e^{\sigma t} \|\nabla \tilde{u}(t)\|^2
+2e^{\sigma t} |\langle g(t, u_t^1)- g(t, u_t^2), \tilde{u}(t) \rangle_{V^{\prime} \times V}|\\
& \quad \quad +
2 e^{\sigma t}|\langle B(u^1(t))-B(u^2(t)), \tilde{u}(t)\rangle_{V^{\prime} \times V}|,
\end{aligned}
\end{align}
where $\sigma$ is defined by \eqref{4.10}.
Thanks to \eqref{2.2}-\eqref{2.3}, we have
\begin{align}\label{4.58}
\begin{aligned}
|\langle B(u^1(t))-B(u^2(t)), \tilde{u}(t)\rangle_{V^{\prime} \times V}|
&\leq
C_1 (\lambda_1^{-1}+1)^3\|\nabla\tilde{u}(t)\|^2 \cdot \|\nabla u^2(t)\| .
\end{aligned}
\end{align}
By the compact imbedding of $L^2(\Omega) \hookrightarrow H^{-1}(\Omega)$ and Young inequality, we conclude
\begin{align}\label{4.59}
\begin{aligned}
&
 |\langle g(t, u_t^1)- g(t, u_t^2), \tilde{u}(t) \rangle_{V^{\prime} \times V}|\\
& \quad \leq
\frac{(\beta+4C_g)(\lambda_1^{-1}+1)}{2}\|\nabla \tilde{u}(t)\|^2+\frac{C_6(\lambda_1^{-1}+1)}{2(\beta+4C_g)}\|g(t, u_t^1)- g(t, u_t^2)\|^2.
\end{aligned}
\end{align}
Combining \eqref{4.57}-\eqref{4.59}, we get
\begin{align}\label{4.60}
\begin{aligned}
&
\frac{\mathrm{d}}{\mathrm{d}t}(e^{\sigma t}\|\tilde{u}(t)\|^2+\alpha^2e^{\sigma t}\|\nabla \tilde{u}(t)\|^2)\\
& \quad \leq
(\lambda_1^{-1}\sigma+\alpha^2 \sigma-2 \nu+(\beta+4C_g)(\lambda_1^{-1}+1))e^{\sigma t} \|\nabla \tilde{u}(t)\|^2\\
& \quad \quad +
2C_1(\lambda_1^{-1}+1)^3 e^{\sigma t} \|\nabla \tilde{u}(t)\|^2 \cdot \|\nabla u^2(t)\|^2
+\frac{C_6(\lambda_1^{-1}+1)}{\beta+4C_g} e^{\sigma t} \|g(t, u_t^1)- g(t, u_t^2)\|^2.
\end{aligned}
\end{align}
Set
\begin{align}\label{4.63}
\eta_1 > \eta_2=\alpha^{-2}\big(2 \nu-\lambda_1^{-1}\sigma-\alpha^2 \sigma-(\beta+4C_g)(\lambda_1^{-1}+1)
-\frac{C_g^2C_6(\lambda_1^{-1}+1)^2}{\beta+4C_g}\big) >0.
\end{align}
Integrating \eqref{4.60} in the interval $[\tau, t]$, and using the assumptions \text{(H1)-(H4)}, we infer that
\begin{align}\label{4.64}
\begin{aligned}
&
\|\nabla \tilde{u}(t)\|^2+\eta_2 e^{-\sigma t} \int_\tau^t e^{\sigma s} \|\nabla \tilde{u}(s)\|^2 \mathrm{d} s\\
& \quad \leq
(1+\lambda_1^{-1}\alpha^{-2}+\frac{C_g^2C_6(\lambda_1^{-1}+1)}{ \alpha^2 (\beta+4C_g)})) e^{-\sigma (t-\tau)} \|(\tilde u_\tau, \tilde{\varphi})\|^2_{E_V^2}\\
& \quad \quad +
2C_1\alpha^{-2}(\lambda_1^{-1}+1)^3  e^{-\sigma t}\int_\tau^t e^{\sigma s} \|\nabla \tilde{u}(s)\|^2 \cdot \|\nabla u^2(s)\|^2 \mathrm{d} s,
\end{aligned}
\end{align}
where $0< \eta_2 < \eta_1$, and $\eta_1$ and $\beta$ are same as those in \eqref{4.13}-\eqref{4.14}.
Similarly to \eqref{4.22}, for any $t-h \geq \tau$,
inequality \eqref{4.64} becomes
\begin{align}\label{4.67}
\begin{aligned}
\int_{t-h}^t \|\nabla \tilde{u}(s)\|^2 \mathrm{d}s
&\leq
\eta_2^{-1}(1+\lambda_1^{-1}\alpha^{-2}+\frac{C_g^2C_6(\lambda_1^{-1}+1)}{ \alpha^2 (\beta+4C_g)})) e^{-\sigma (t-\tau-h)} \|(\tilde u_\tau, \tilde{\varphi})\|^2_{E_V^2}\\
& \quad +
2C_1 \eta_2^{-1} \alpha^{-2}(\lambda_1^{-1}+1)^3  e^{-\sigma (t-h)}\int_\tau^t e^{\sigma s} \|\nabla \tilde{u}(s)\|^2 \cdot \|\nabla u^2(s)\|^2 \mathrm{d} s.
\end{aligned}
\end{align}
Combining \eqref{4.64} and \eqref{4.67}, we arrive at
\begin{align}\label{4.68}
\begin{aligned}
&
\|U(t, \tau(u_\tau^1, \varphi^1))-U(t, \tau(u_\tau^2, \varphi^2))\|^2_{E_V^2}
=
\|\nabla \tilde{u}(t)\|^2+\int_{t-h}^t \|\nabla \tilde{u}(s)\|^2 \mathrm{d}s\\
& \quad \leq
(1+\eta_2^{-1})(1+\lambda_1^{-1}\alpha^{-2}+\frac{C_g^2C_6(\lambda_1^{-1}+1)}{ \alpha^2 (\beta+4C_g)})) e^{-\sigma (t-\tau-h)} \|(\tilde u_\tau, \tilde{\varphi})\|^2_{E_V^2}\\
& \quad \quad +
2C_1 \alpha^{-2} (1+\eta_2^{-1}) (\lambda_1^{-1}+1)^3  e^{-\sigma (t-h)}\int_\tau^t e^{\sigma s} \|\nabla \tilde{u}(s)\|^2 \cdot \|\nabla u^2(s)\|^2 \mathrm{d} s.
\end{aligned}
\end{align}
Furthermore, let $T=t-\tau$ and
\begin{align}\label{4.69}
\psi_{t, T}\big((u^1, u_t^1), (u^2, u_t^2)\big)=2C_1 \alpha^{-2} (1+\eta_2^{-1}) (\lambda_1^{-1}+1)^3  e^{-\sigma (t-h)}\int_\tau^t e^{\sigma s} \|\nabla \tilde{u}(s)\|^2 \cdot \|\nabla u^2(s)\|^2 \mathrm{d} s.
\end{align}
Then by Definition  \ref{defn2.6}, Lemmas \ref{lem2.1}, \ref{lem4.1}, \ref{lem4.2} and \ref{lem4.3}, it follows that $\psi_{t, T}\big((u^1, u_t^1), (u^2, u_t^2)\big)$ is a contractive function. For any constant $C_8>0$, taking
$$
\tau=t-h-\frac{1}{\sigma}\frac{(1+\eta_2^{-1})(\lambda_1^{-1}\alpha^{-2}+1+\frac{C_g^2C_6(\lambda_1^{-1}+1)}{ \alpha^2 (\beta+4C_g)}))  \|(\tilde u_\tau, \tilde{\varphi})\|^2_{E_V^2}}{C_8},
$$
then it can be seen that the process $\{U(t, \tau)\}_{t \geq \tau}$ of problem \eqref{1.1} is pullback $\mathcal{D}_\sigma^{E^2_V}$-asymptotically compact in $E_V^2$.
\end{proof}

From the above proofs, it follows the following theorem about the pullback $\mathcal{D}_\sigma^{E^2_V}$-attractors, which is one of the main results of this paper.

\begin{thm}\label{thm4.1}
Under the assumptions of Lemma \ref{lem4.2}, assume that the function $f(x, t)$ satisfies \eqref{4.26}, then there exists a unique pullback $\mathcal{D}_\sigma^{E^2_V}$-attractor $\mathcal{A}_\sigma^{E^2_V}=\{A_\sigma^{E^2_V}(t): t \in \mathbb{R}\}$ of problem $(1.1)$ in $E^2_V$.
\end{thm}

\begin{proof}
From Definitions \ref{defn2.7} and \ref{defn4.1}, Theorem \ref{thm2.1}, Lemmas \ref{lem4.1}-\ref{lem4.4}, it follows the existence and uniqueness of the above pullback $\mathcal{D}_\sigma^{E^2_V}$-attractor $\mathcal{A}_\sigma^{E^2_V}$.
\end{proof}

\section{Regularity of pullback attractors}

\quad \,  In the section, as in \cite{CP09, WZL18, ZS15}, we shall establish the regularity of pullback attractors for non-autonomous system \eqref{1.1}.

\begin{lem}\label{lem6.1}
Under the assumptions of Theorem \ref{thm4.1}, then $\mathcal{D}_\sigma^{E_V^2}$ is bounded in $E^2_{D(A)}$.
\end{lem}

\begin{proof}
Since $L^2(\Omega) \hookrightarrow H^{-1}(\Omega)$ is dense (see \cite{A75, E98}), for any $f(x, t) \in L_{l o c}^2(\mathbb{R} ; V^{\prime})$, there exists a function $f^\theta(x, t) \in L_{l o c}^2(\mathbb{R} ; H)$ such that
\begin{align}\label{6.77}
\|f-f^\theta\| < \xi,
\end{align}
where $\xi > 0$ is a constant.

Fixed $\tau \in \mathbb{R}$ and suppose $(u_\tau, \varphi) \in \mathcal{D}_\sigma^{E_V^2}$, then we can decompose the solution $U(t, \tau)(u_\tau, \varphi)=(u(t), u_t)$ into the sum
\begin{align}\label{6.78}
U(t, \tau)(u_\tau, \varphi)=U_1(t, \tau)(u_\tau, \varphi)+U_2(t, \tau)(u_\tau, \varphi),
\end{align}
with $U_1(t, \tau)(u_\tau, \varphi)=(v(t), 0)$ and $U_2(t, \tau)(u_\tau, \varphi)=(w(t), u_t)$ satisfying the following equations, respectively,
\begin{align}\label{6.79}
\begin{cases}
\partial_t v-\alpha^2\partial_t \Delta v-\nu \Delta v+(v \cdot \nabla)v+\nabla p=f(t)-f^{\theta}(t) & \text { in } (\tau, +\infty) \times  \Omega, \\
\nabla \cdot v=0 & \text { in } (\tau, +\infty) \times  \Omega, \\
v(\tau, x)=u(\tau, x)=u_\tau & \text { in }\Omega, \\
v(\tau+t, x)=0, &  \text{ in } (-h, 0) \times \Omega,
\end{cases}
\end{align}
and
\begin{align}\label{6.80}
\begin{cases}
\partial_t w-\alpha^2\partial_t \Delta w-\nu \Delta w+(w \cdot \nabla)w+\nabla p= g(t, u_t)+f^{\theta}(t) & \text { in } (\tau, +\infty) \times  \Omega, \\
\nabla \cdot w=0 & \text { in } (\tau, +\infty) \times  \Omega, \\
w(\tau, x)=0 & \text { in }\Omega, \\
w(\tau+t, x)=\varphi(t, x), &  \text{ in } (-h, 0) \times \Omega.
\end{cases}
\end{align}
Multiplying \eqref{6.79}$_1$ by $Av \in H$, integrating it in $\Omega$ and multiplying both sides by $e^{\sigma t}$, we obtain
\begin{align}\label{6.81}
\frac{\mathrm{d}}{\mathrm{d}t} (\|\nabla v(t)\|^2+\alpha^2 \|A v(t)\|^2)
\leq
-2\nu\|Av(t)\|^2+2|b(v(t), v(t), Av(t))|+2|(f(t)-f^\theta(t), Av(t))|.
\end{align}
Besides, from \eqref{2.4}, the Poincar\'{e} and Young inequalities, we can estimate
\begin{align}\label{6.83}
\begin{aligned}
2|b(v(t), v(t), Av(t))|
\leq
4C_2 (\lambda_1^{-1}+1)^3\|\nabla v(t)\|^6+\frac{3C_2}{8}\|Av(t)\|^2,
\end{aligned}
\end{align}
and
\begin{align}\label{6.84}
\begin{aligned}
2|(f(t)-f^\theta(t), Av(t))|
\leq
\frac{\xi^2}{(\beta+4C_g)(\lambda_1^{-1}+1)}
+(\beta+4C_g)(\lambda_1^{-1}+1)\|Av(t)\|^2.
\end{aligned}
\end{align}
Similar to Lemma \ref{lem5.2}, we have
\begin{align}\label{6.147}
\begin{aligned}
\int_{t-h}^t \|Av(s)\|^2 \mathrm{d}s
&\leq
\eta_5^{-1}({\lambda_1^{-1}\alpha^{-2}+1})e^{-\sigma (t-\tau-h)}\|(u_\tau, \varphi)\|^2_{E_{D(A)}^2}
+\frac{\xi^2e^{\sigma h}}{\eta_5\alpha^2(\beta+4C_g)(\lambda_1^{-1}+1)}\\
& \quad +
4C_2\eta_5^{-1}\alpha^{-2}(\lambda_1^{-2}+\lambda_1^{-1})^3\mathop{\max}_{s \in [\tau, t]}\|Av(s)\|^6,
\end{aligned}
\end{align}
where
$$
\eta_1 >  \eta_5=\alpha^{-2}(2\nu -\alpha^2\sigma-\lambda_1^{-1}\sigma-(\beta+4C_g)(\lambda_1^{-1}+1)-\frac{3C_2}{8}) > 0.
$$
From \eqref{6.147}, we know that
\begin{align}\label{6.148}
\begin{aligned}
&
\|U_1(t, \tau)(u_\tau, \varphi)\|^2_{E_{D(A)}^2}
=\|(v(t), 0)\|^2_{E_{D(A)}^2}=\|Av(t)\|^2\\
&\leq
({\lambda_1^{-1}\alpha^{-2}+1})e^{-\sigma (t-\tau)}\|(u_\tau, \varphi)\|^2_{E_{D(A)}^2}
+\frac{\xi^2}{\alpha^2(\beta+4C_g)(\lambda_1^{-1}+1)}\\
& \quad +
4C_2\alpha^{-2}(\lambda_1^{-2}+\lambda_1^{-1})^3\mathop{\max}_{s \in [\tau, t]}\|Av(s)\|^6.
\end{aligned}
\end{align}
In the same way, multiplying \eqref{6.80}$_1$ by $Aw \in H$, integrating the resultant in $\Omega$ and multiplying both side by $e^{\sigma t}$, we conclude
\begin{align}\label{6.150}
\begin{aligned}
& \frac{\mathrm{d}}{\mathrm{d}t}( e^{\sigma t}\|\nabla w(t)\|^2+\alpha^2 e^{\sigma t} \|A w(t)\|^2)\\
& \quad \leq
(\alpha^2\sigma+\sigma\lambda_1^{-1}-2\nu)e^{\sigma t}\|Aw(t)\|^2+2e^{\sigma t}|b(w(t), w(t), Aw(t))|\\
& \quad \quad +
2e^{\sigma t}|(f^\theta(t), Aw(t))|+2e^{\sigma t}|(g(t, u_t), Aw(t))|.
\end{aligned}
\end{align}
By simple calculations, we arrive at
\begin{align}\label{6.151}
\begin{aligned}
2|b(w(t), w(t), Aw(t))|
\leq
\frac{3C_2}{8}\|Aw(t)\|^2
+
2C_2 (\lambda_1^{-1}+1)^3\|\nabla w(t)\|^6,
\end{aligned}
\end{align}
\begin{align}\label{6.152}
\begin{aligned}
2|(f^\theta(t), Aw(t))|
\leq
\frac14\|Aw(t)\|^2
+4\|f^\theta(t)\|^2_{V^{\prime}},
\end{aligned}
\end{align}
and
\begin{align}\label{6.153}
\begin{aligned}
2|(g(t, u_t), Aw(t))|
\leq
(\beta+4C_g)(\lambda_1^{-1}+1)\|Aw(t)\|^2
+\frac{1}{(\beta+4C_g)(\lambda_1^{-1}+1)}\|g(t, u_t)\|^2.
\end{aligned}
\end{align}
Eventually, inserting \eqref{6.151}-\eqref{6.153} to \eqref{6.150}, we end up with
\begin{align}\label{6.155}
\begin{aligned}
\int_{t-h}^t \|Aw(s)\|^2 \mathrm{d}s
&\leq
\eta_6^{-1}({\lambda_1^{-1}\alpha^{-2}+1}+\frac{1}{\alpha^1(\beta+4C_g)(\lambda_1+1)})e^{-\sigma (t-\tau)}\|(u_\tau, \varphi)\|^2_{E_{D(A)}^2}\\
& \quad +
2\eta_6^{-1}\alpha^{-2}e^{-\sigma t}\int_\tau^t e^{\sigma s} \|f^\theta(s)\|^2_{V^{\prime}} \mathrm{d}s\\
& \quad +
4C_2\eta_6^{-1}\alpha^{-2}(\lambda_1^{-2}+\lambda_1^{-1})^3\mathop{\max}_{s \in [\tau, t]}\|Aw(s)\|^6,
\end{aligned}
\end{align}
where
$$
\eta_1 >  \eta_6=\alpha^{-2}(2\nu -\alpha^2\sigma-\lambda_1^{-1}\sigma-(\beta+4C_g)(\lambda_1^{-1}+1)-\frac{3C_2+2}{8}-\frac{\lambda_1^{-1}}{\beta+4C_g}) > 0.
$$
From \eqref{6.77}, \eqref{6.155} and \eqref{6.155}, we derive that
\begin{align}\label{6.156}
\begin{aligned}
&
\|U_2(t, \tau)(u_\tau, \varphi)\|^2_{E_{D(A)}^2}
=\|Aw(t)\|^2+\int_{t-h}^t \|Aw(s)\|^2 \mathrm{d} s\\
& \quad \leq
(1+\eta_6^{-1}e^{\sigma h})({\lambda_1^{-1}\alpha^{-2}+1}+\frac{1}{\alpha^1(\beta+4C_g)(\lambda_1+1)})e^{-\sigma (t-\tau)}\|(u_\tau, \varphi)\|^2_{E_{D(A)}^2}+c \rho_\sigma^{\prime}(t)\\
& \quad \leq
R_5(t),
\end{aligned}
\end{align}
where $R_5(t)=c \rho_\sigma^{\prime}(t)$.
It follows from \eqref{6.148} and \eqref{6.156} that for any $t \in \mathbb{R}$,
\begin{align}\label{6.157}
\textrm{dist}_{E^2_{D(A)}}(\mathcal{D}_\sigma^{E_V^2}, \overline{\mathcal{B}}_{E^2_{D(A)}}(R_5))
\leq
c e^{-\sigma(t-\tau)} \rightarrow 0, \enskip \text{as} \, \tau \rightarrow -\infty.
\end{align}
Consequently, \eqref{6.157} implies the pullback $\mathcal{D}_\sigma^{E_V^2}$-attractors $\mathcal{A}_\sigma^{E^2_V}$ is bounded in $E_{D(A)}^2$.
\end{proof}

\section{Invariant measure on the pullback attractors}

\quad \,  The aim of this section is to employ the theory of {\L}ukaszewicz and Robinson \cite{LR14} to prove the unique existence of invariant measure on the pullback $\mathcal{D}_\sigma^{E_V^2}$-attractor $\mathcal{A}_\sigma^{E^2_V}$.
We first cite the definition of generalized Banach limits.

\begin{defn}(\cite{FMRT01, LR14}).
A generalized Banach limit is any linear functional, which we denote by $\operatorname{LIM}_{T \rightarrow \infty}$, defined on the space of all bounded real-valued functions on $[0,+\infty)$ that satisfies\\
(i) $\operatorname{LIM}_{T \rightarrow \infty} \zeta(T) \geq 0$ for nonnegative functions $\zeta$;\\
(ii) $\operatorname{LIM}_{T \rightarrow \infty} \zeta(T)=\lim \limits_{T \rightarrow \infty} \zeta(T)$ if the usual limit $\lim \limits_{T \rightarrow \infty} \zeta(T)$ exists.
\end{defn}

\begin{rem}
Notice that we consider the ``pullback" asymptotic behavior and we require generalized limits as $\tau \rightarrow-\infty$. For a given real-valued function $\zeta$ defined on $(-\infty, 0]$ and a given Banach limit $\operatorname{LIM}_{t \rightarrow+\infty}$, we define $\operatorname{LIM}_{t \rightarrow-\infty} \zeta(t)=$ $\operatorname{LIM}_{t \rightarrow+\infty} \zeta(-t)$.
\end{rem}




\begin{lem}\label{lem5.1}
Assume $f \in L_{l o c}^2(\mathbb{R} ; V^{\prime})$ satisfying \eqref{4.26}, $(u_\tau, \varphi) \in E_V^2$ and $g: \mathbb{R} \times C_V \rightarrow L^2(\Omega)$ satisfies \text{(H1)-(H4)}. Then for every $(u_\tau, \varphi) \in E_V^2$ and every $t \in \mathbb{R}$ the $E_V^2$-valued function $\tau \longmapsto U(t, \tau)(u_\tau, \varphi)$ is bounded on $(-\infty, t]$.
\end{lem}

\begin{proof}
Let $(u_\tau, \varphi) \longmapsto U(t, \tau)(u_\tau, \varphi)$ and $t \in \mathbb{R}$ be given.
Then \eqref{94} yields to
\begin{align}\label{93}
\begin{aligned}
\|U(t,\tau)(u_\tau, \varphi)\|^2_{E_V^2}
\leq&
c(1+\eta_1^{-1}e^{\sigma h}) \big(||(u_\tau,\varphi)||^2_{E_V^2}
+ e^{-\sigma t} \int_{- \infty}^t e^{\sigma s}||f(s)||^2_{V^{\prime}}\mathrm{d} s\big), \enskip \tau \in(-\infty, t],\\
\end{aligned}
\end{align}
where $c> 0$ is a constant that depends only on $\alpha$, $\lambda_1$, $C_g$ and $\nu$. Thus, we can infer that the right-hand side of \eqref{93} gives the boundedness of the $E_V^2$-valued function $(\tau \longmapsto U(t, \tau)(u_\tau, \varphi)$. 
\end{proof}


\begin{lem}\label{lem100}
Let $u_\tau, v_\tau \in V$, and $u(t)=u(t ; \tau, u_\tau, \varphi), v(t)=v(t ; \tau, v_\tau, \varphi)$ be the corresponding solutions to problem \eqref{1.1}. Then for the all $t \in[0, T]$, the following equality holds
there exists some positive constant $M(c_1, c_2, N)$ such that for any $t>\tau$,
\begin{align}
\|\nabla u(t)-\nabla v(t)\|
\leq
\alpha^{-2}(v+C_1 \max _{s \in[0, T]} (\|\nabla u(s) \|+ \|\nabla v(s) \| )+2 C_g) \int_0^t\|\nabla \tilde{u}(s)\| d s,
\enskip \forall t \in[0, T].
\end{align}
\end{lem}

\begin{proof}
Let $u(t)$ and $v(t)$ be two solutions of problem \eqref{1.1} corresponding to the initial data $u_\tau$ and $v_\tau$, respectively.
Denote $\tilde{u}(t)=u(t)-v(t)$ and one can check that $\tilde{u}(0, x)=0$ in $\Omega$ and  $\tilde{u}(t, x)=0$ in $(-h, 0) \times \Omega$.
It is easy to deduce that $\tilde{u}(t)$ satisfies the following equality for all $t>0$
\begin{align}\label{90}
\begin{aligned}
&
\|\tilde{u}(t)\|^2+\alpha^2\|\nabla \tilde{u}(t)\|^2+2 v \int_0^t\| \nabla\tilde{u}(s)\|^2 d s+2 \int_0^t \langle B (u(s))-B(v(s)), \tilde{u}(s)\rangle d s \\
& \quad =
2 \int_0^t \langle g(s, u_s)-g(s, v_s), \tilde{u}(s) \rangle d s. 
\end{aligned}
\end{align}
Besides, by \eqref{2.2}, we can obtain that 
\begin{align}\label{89}
\begin{aligned}
&
\|B(u(s))-B(v(s))\|_{V^{\prime}}  
\leq  C_1(\nabla \|u(s)\|+ \nabla \|v(s)\|)\|\nabla u(s)-\nabla v(s)\|.
\end{aligned}
\end{align}
Thus, if we fix an arbitrary $T>0$, and denote $R_T=C_1 \max _{s \in[0, T]} (\|\nabla u(s) \|+ \|\nabla v(s) \| )$, we have
$$
\|B(u(s) )-B (v(s) ) \|_{V^{\prime}}
\leq
R_T \|\nabla u(s)-\nabla v(s) \|, \enskip \forall s\in[0, T] .
$$
Then, by \text{(H3)}, we deduce that
\begin{align}\label{88}
\begin{aligned}
2 \int_0^t \langle g (s, u_s)-g (s, v_s ), \tilde{u}(s) \rangle d s
&\leq
2 \int_0^t \|g (s, u_s )-g (s, v_s) \|_{V^{\prime}}\|\nabla \tilde{u}(s)\| d s \\
& \leq
2 (\int_0^t \|g (s, u_s )-g (s, v_s ) \|_{V^\prime}^2 d s )^{1 / 2} (\int_0^t\|\nabla \tilde{u}(s)\|^2 d s )^{1 / 2} \\
&\leq
2C_g (\int_{-h}^t \|\nabla u(s)-\nabla v(s) \|^2 d s )^{1 / 2} (\int_0^t\|\nabla \tilde{u}(s)\|^2 d s )^{1 / 2} \\
&=
2 C_g \int_0^t\|\nabla \tilde{u}(s)\|^2 d s.
\end{aligned}
\end{align}
As $\|A \tilde{u}(s)\|_{V^{\prime}}=\|\tilde{u}(s)\|$, from \eqref{90}-\eqref{88}, we deduce that
\begin{align}
\|\nabla \tilde{u}(t)\|
\leq
\alpha^{-2}(v+C_1 \max _{s \in[0, T]} (\|\nabla u(s) \|+ \|\nabla v(s) \| )+2 C_g) \int_0^t\|\nabla \tilde{u}(s)\| d s.  
\end{align}
\end{proof}

\begin{lem}\label{lem5.2}
Assume $f \in L_{l o c}^2(\mathbb{R} ; V^{\prime})$ satisfying \eqref{4.26}, $(u_\tau, \varphi) \in E_V^2$ and $g: \mathbb{R} \times C_V \rightarrow L^2(\Omega)$ satisfies \text{(H1)-(H4)}. Then the process $\{U(t, \tau)\}_{t \geq \tau}$ is $\tau$-continuous in space $E_V^2$.
\end{lem}

\begin{proof}
Consider any $(u_\tau, \varphi) \in E_V^2$ and $t \in \mathbb{R}$. We shall prove that for any $\epsilon>0$, there exists some $\delta=\delta(\epsilon)>0$, such that if $r<t, s<t$ and $|r-s|<\delta$, then $\|U(t, r) (u_\tau, \varphi)-U(t, s) (u_\tau, \varphi)\|_{E_V^2}<\epsilon$. We assume that $r<s$ without loss of generality. Then employing Lemma \ref{lem4.1} and the property of the continuous process referred by Lemma \ref{lem3.1}, we have
\begin{align}\label{92}
\begin{aligned}
& \|U(t, r)(u_\tau, \varphi)-U(t, s) (u_\tau, \varphi)\|_{E_V^2}^2
=
\|U(t, s) U(s, r)(u_\tau, \varphi)-U(t, s) U(r, r)(u_\tau, \varphi)\|_{E_V^2}^2 \\
& \quad =
\|\nabla U(s, r)u_\tau- \nabla U(r, r)u_\tau\|^2
+\int_{-h}^0 \| U(s, r)\nabla \varphi(\theta)-U(r, r)\nabla \varphi(\theta)\|^2 \mathrm{d} \theta\\
& \quad \leq
\|\nabla U(s, r)u_\tau- \nabla U(r, r)u_\tau\|^2
+
h \cdot \sup \limits_{\theta \in [-h,0]}  \|\nabla u(s+\theta)-\nabla u(r+\theta)\|^2 \\
\end{aligned}
\end{align}
From statement \eqref{4.9}, Lemmas \ref{lem4.1} and \ref{lem100}, we directly conclude that the right hand side of inequality \eqref{92} is as small as needed if $|r-s|$ is small enough. Therefore, the proof is conpleted.
\end{proof}

At this stage, we take Theorem \ref{thm4.1}, Lemma 4.1, Lemma 4.4 and the definition of $\tau$- continuous (see \cite{LR14}) into account and give the main result of this article as follow.

\begin{thm}\label{thm5.1}
Assume $f \in L_{l o c}^2(\mathbb{R} ; V^{\prime})$ satisfying \eqref{4.26}, $(u_\tau, \varphi) \in E_V^2$ and $g: \mathbb{R} \times C_V \rightarrow L^2(\Omega)$ satisfies \text{(H1)-(H4)}.
Let $\{U(t, \tau)\}_{t \geq \tau}$ be the process associated to the solution operators of equation \eqref{91} and $\mathcal{A}_\sigma^{E^2_V}=\{A_\sigma^{E^2_V}(t): t \in \mathbb{R}\}$ be the pullback $\mathcal{D}_\sigma^{E_V^2}$-attractor obtained in Theorem \ref{thm4.1}. Fix a generalized Banach limit $\operatorname{LIM}_{T \rightarrow \infty}$ and let $\varrho: \mathbb{R} \mapsto E_V^2$ be a continuous map such that $\varrho(\cdot) \in \mathcal{D}_\sigma^{E_V^2}$. Then there exists a unique family of Borel probability measures $\{\mu_t\}_{t \in R}$ in space $E_V^2$ such that the support of the measure $\mu_t$ is contained in $\mathcal{A}_\sigma^{E^2_V}$ and
\begin{align*}
\operatorname{LIM}_{\tau \rightarrow-\infty} \frac{1}{t-\tau} \int_\tau^t \Phi(U(t, s) \varrho(s)) \mathrm{d} s
&=
\int_{A_{\sigma}^{E_V^2}(t)} \Phi(u) \mathrm{d} \mu_t(u)
=
\int_{E_V^2} \Psi(u) \mathrm{d}\mu_t(u) \\
&=
\operatorname{LIM}_{\tau \rightarrow-\infty} \frac{1}{t-\tau} \int_\tau^t \int_{E_V^2} \Psi(U(t, s) u) \mathrm{d} \mu_s(u) \mathrm{d} s,
\end{align*}
for any real-valued continuous functional $\Phi$ on $E_V^2$. In addition, $\mu_t$ is invariant in the sense that
\begin{align*}
\int_{A_{\sigma}^{E_V^2}(t)} \Phi(u) \mathrm{d} \mu_t(u)
=
\int_{A_{\sigma}^{E_V^2}(\tau)} \Phi(U(t, \tau) u) \mathrm{d} \mu_\tau(u), \enskip t \geq \tau .
\end{align*}
\end{thm}

\begin{proof}
From Theorem \ref{thm2.1} and Lemma \ref{lem4.1}, the solution operators of the problem \eqref{1.1} generate a continuous process $\{U(t, \tau)\}_{t \geq \tau}$ on the space $E_V^2$. Theorem \ref{thm4.1} shows that $\{U(t, \tau)\}_{t \geq \tau}$ possesses a pullback attractor in $E_V^2$. 
Combining the above facts, Lemmas \ref{lem5.1} and \ref{lem5.2} and the abstract result of \cite[Theorem 3.1]{LR14}, we obtain Theorem \ref{thm5.1}.
\end{proof}

\hskip 3mm

\noindent\textbf{Acknowledgement}  \enskip This work was supported by the NNSF of China with contract No.12171082, the TianYuan Special Funds of the National Natural Science Foundation of China (Grant No. 12226403), the fundamental research funds for the central universities with contract numbers 2232022G-13, 2232023G-13 and by a grant from Science and Technology Commission of Shanghai Municipality.

\end{document}